\documentclass[12pt]{amsart}
\usepackage{amsmath,amsthm,amsfonts,amssymb}

\newtheorem{theorem}{Theorem}[section] 
 
\newtheorem{lemma}[theorem]{Lemma}

\newtheorem{prop}[theorem]{Proposition}
\newtheorem{remark}[theorem]{Remark}

\numberwithin{equation}{section}

\usepackage{color}

\newcommand{\lam}{\lambda}
\newcommand{\eps}{\varepsilon}

\newcommand{\x}{\mathbf{x}}
\newcommand{\z}{\mathbf{z}}
\newcommand{\e}{\mathbf{e}}
\newcommand{\f}{\mathbf{f}}
\newcommand{\g}{\mathbf{g}}
\newcommand{\bu}{\mathbf{u}}
\newcommand{\bk}{\mathbf{k}}
\newcommand{\y}{\mathbf{y}}
\newcommand{\w}{\mathbf{w}}
\newcommand{\X}{\mathbf{X}}
\newcommand{\bv}{\mathbf{v}}
\newcommand{\0}{\mathbf{0}}

\newcommand{\I}{\mathbb{I}}
\newcommand{\R}{\mathbb{R}}
\newcommand{\Z}{\mathbb{Z}}

\begin{document}

\title{The Effective Radius of Self Repelling Elastic Manifolds}
\author{Carl Mueller}
\address{Carl Mueller
\\Department of Mathematics
\\University of Rochester
\\Rochester, NY  14627}
\email{carl.e.mueller@rochester.edu}
\author{Eyal Neuman}
\address{Eyal Neuman
\\Department of Mathematics
\\Imperial College
\\London, UK}
\email{e.neumann@imperial.ac.uk}
\thanks{CM was partially supported by Simons Foundation grant 513424.} 

\begin{abstract} We study elastic manifolds with self-repelling 
terms and estimate their effective radius. This class of 
manifolds is modelled by a self-repelling vector-valued Gaussian free field 
with Neumann boundary conditions over the domain $[-N,N]^d\cap \mathbb{Z}^d$, 
that takes values in $\mathbb{R}^d$. Our main result states that in two 
dimensions ($d=2$), the effective radius $R_N$ of the manifold is
approximately $N$.  This verifies the conjecture of Kantor, Kardar and Nelson \cite{Kardar_87} up to a logarithmic correction. Our results in $d\geq 3$ give a similar lower bound on 
$R_N$ and an upper of order $N^{d/2}$.  This result implies that 
self-repelling elastic manifolds undergo a substantial stretching at any 
dimension.   
\end{abstract}

\maketitle

\section{Introduction}
\label{section:introduction}

\subsection{Motivation}

The Gaussian free field (GFF) has become a central object in probability, 
studied for its own sake and with connections to several areas of physics.  
In pure probability, GFF plays the role of Brownian motion with 
multidimensional time.   In Euclidean field theory, GFF represents a field 
of noninteracting particles, with interacting models such as the $\phi^4$ 
field theory arising through a change of measure, often requiring 
renormalization.  In fact GFF over $\R^d$ with $d\ge2$ is valued
in the space of Schwartz distributions and cannot be realized as a function, 
so its fourth power $\phi^4$ is undefined in the usual sense 
(see Biskup \cite{Bis2020}).

In this paper we take inspiration from a different set of physical models, 
elastic manifolds (see Mezard and Parisi \cite{MP92} and Balents and Fisher
\cite{BF93}).  Here GFF is vector-valued, and represents the position of a random surface.  
To avoid using Schwartz distributions as mentioned above, we discretize the 
domain of GFF, and use the notation DGFF (discrete GFF) for the 
resulting model.  

If the domain of DGFF is one dimensional, then we have a well-known model of 
a random polymer.  In this context, it is common to include a self-repelling 
term which reflects the fact that different parts of the polymer cannot 
occupy the same position.  A typical object of study is the end-to-end 
distance of such a polymer, or the closely related concept of effective 
radius.  There is a vast literature on such problems, see 
Bauerschmidt, Duminil-Copin, Goodman, and Slade \cite{BDGS12} and the 
included citations.  

In the context of random surfaces, GFF's are also called elastic manifolds.  
The purpose of this paper is to study elastic manifolds with self-repelling 
terms, and to estimate the effective radius in the case. Self-repelling elastic 
manifolds were first introduced by Kantor, Kardar and Nelson in \cite{Kardar_86} as generalizations of 
polymer models to higher dimensions, in order to capture the behaviour of sheets of covalently bonded atoms and of polymerized lipid surfaces, among others. See \cite{Kardar_86,Kardar_87,Kantor_87,Nelson_2004} and references therein for additional details. 
To our knowledge no one has studied these models in the mathematical literature. In this class of models, we define an $\mathbb{R}^D$-valued DGFF over $[-N,N]^d \cap \mathbb{Z}^d$, and use 
Neumann boundary conditions since these are most closely tied to a random 
surface with free boundary conditions.  In two dimensions, that is $d=D=2$, 
we get fairly tight bounds on the effective radius of the manifold $R_N$, which is proportional to $N$ in the upper 
and lower bounds, up to logarithmic corrections.  Our results in higher 
dimensions are not as sharp, as we derive a lower bound on $R_N$ that is proportional to $N$, but the upper bound is of order $N^{d/2}$. This proves however that  
self-repelling elastic manifolds experience a substantial stretching at any 
dimension.  

We remark that for the case where $D=d=1$ we expect $R_N$ to have the same asymptotic behaviour as a one dimensional self-repelling polymer, that is $R_N \sim N$. This result can be derived by the argument of the proof of Theorem \ref{th:main}, at least up to a logarithmic constant. We leave the details of the proof to the reader.

The case where $D<d$ was studied by the authors in a followup paper \cite{M-N-22}. 
It was proved that when 
the dimension of the domain is $d=2$ and the dimension of the range is $D=1$, 
the effective radius $R_N$ of the manifold is approximately $N^{4/3}$. The 
results for the case $d \geq 3$ and $D <d$ give a lower bound on $R_N$ of 
order $N^{\frac{1}{D} \left(d-\frac{2(d-D)}{D+2} \right)}$ and an upper 
bound proportional to $N^{\frac{d}{2}+\frac{d-D}{D+2}}$. The results of \cite{M-N-22} imply 
that self-repelling elastic manifolds with a low dimensional range undergo a 
significantly stronger stretching than in the case $d=D$, which is studied in this paper.

The remaining case, $D>d$, looks to be much harder. For example, consider the case where the domain of the self-repelling DGFF is 
$\{0,\dots,N\}$ and it takes values in $\R^D$.  For $D=2,3,4$ the behavior 
of the effective radius of the self-repelling polymer as $N\to\infty$ is 
still unknown,  although we have good information for $D=1$ and for $D>4$.  
See page 400 of Bauerschmidt, Duminil-Copin, Goodman, and Slade \cite{BDGS12} 
and also Bauerschmidt, Slade, and Tomberg, and Wallace \cite{BSTW17}.  
If $D$ is large enough, then for self-avoiding walks, the lace 
expansion can be used.  For DGFF however, there appears to be no analogue of the lace 
expansion.

We briefly compare the mathematical results for self-repelling elastic 
manifolds, which were described above, to conjectures which are available in 
the physical literature. Most of such conjectures are based on the so called 
Flory's argument. We refer to Section \ref{appendix} for a heuristic 
derivation of such result for $D=d=2$. In general this heuristic argument 
suggests that 
\begin{equation} \label{flory} 
R_N \sim N^{\nu}, \quad \textrm{ with } \nu= \frac{d+2}{D+2}. 
\end{equation} 
For the cases where $D=1, d=2$ and $D=2,d=2$ we can confirm \eqref{flory} up to a logarithmic correction by Theorem \ref{th:main} in this paper and by the results in \cite{M-N-22}. For the case where $D=3$ and $d=2$ the physical literature is inconclusive as simulations suggest that $R_N$ grows linearly in $N$ in contradiction to \eqref{flory}. Not much is known about other cases besides from these heuristic results and some arguably imprecise results using $\varepsilon$-expansions (see \cite{Kardar_87}). We refer to Chapter 10.5.2 of \cite{plischke} and references therein for additional information about Flory's argument for self-repelling manifolds and for related results on simulations.
\subsection{Setup}
Now we give more precise definitions; the reader can also consult
Biskup \cite{Bis2020}.  Since the standard definition of DGFF involves Dirichlet 
rather than Neumann boundary conditions, we give details of its construction 
below. In the following, ordinary letters such as $x,u$ take values in 
$\R$ or $\Z$, while boldface letters such as $\x,\bu$ take 
values in $\R^d$ for $d\ge2$.  

Fix $d\ge2, N\ge1$ and define our parameter set as follows.  
\begin{equation*}
S^d_{N}:=[-N,N]^d \cap \Z^d. 
\end{equation*}
Note that
\begin{equation*}
S^1_{N}:=\{-N,\dots,N\}.
\end{equation*}
Thus $S^d_N$ is a cube in $\Z^d$ centered at the origin.  Let 
$\mathcal{N}=\mathcal{N}_{N,d}$ be the set of unordered nearest neighbor 
pairs in $S^d_N$.  

Now we define the discrete Neumann Laplacian $\Delta=\Delta_{N,d,D}$ as 
follows.  Given functions $\f,\g:S_N^d\to\R^D$ we define
the energy $H$ and the inner product $(\cdot,\cdot)=(\cdot,\cdot)_{N,d,D}$ 
as follows,
\begin{equation*}
\begin{split}
H(\f,\g) &= \sum_{(\x,\y)\in\mathcal{N}}
                  \big(\f(\x)-\f(\y)\big)\cdot\big(\g(\x)-\g(\y)\big),  \\
(\f,\g) &= \sum_{\x\in S_N^d}\f(\x)\cdot\g(\x).
\end{split}
\end{equation*}
We will usually write $H(\f)$ instead of $H(\f,\f)$.  
Define the operator $\Delta=\Delta_{N,d,D}$ on such functions by the 
requirement
\begin{equation*}
H(\f,\g) = -(\f,\Delta\g).
\end{equation*}
If $D=1$, we simply write $f,g$ instead of $\f,\g$.  

We claim that $\Delta$ is a self-adjoint operator defined 
pointwise as follows.  Given $\f:S_N^d\to\R^D$, we first extend the 
domain of $\f$ to $S_{N+1}^d$.  If $\x,\y$ are nearest neighbors in 
$\Z^d$ with $\x\in S_N^d$ and $\y\not\in S_N^d$, define 
$\f(\y):=\f(\x)$.  If $\y\in S_{N+1}^d\setminus S_{N}^d$ but $\y$ is not a 
nearest neighbor of any point in $S_N^d$, let $\f(\y):=0$.  We leave it to 
the reader to use summation by parts to verify that for $\x\in S_N^d$,
\begin{equation} \label{eq:def-Laplacian}
\Delta \f(\x) = \sum_{\y\sim \x}\left[\f(\y) - 2d\cdot \f(\x)\right],
\end{equation}
where $\x\sim\y$ means that $\x,\y$ are nearest neighbors.  

In the case $D=1$, since $-\Delta$ is a self-adjoint operator on a 
finite-dimensional space, there exists a finite index set $\I=\I_{N,d}$ to 
be defined later, and an orthonormal basis of eigenfunctions 
$(\varphi_\bk)_{\bk\in\I}$ with corresponding eigenvalues 
$(\lambda_\bk)_{\bk\in\I}$.  We can assume without loss of generality that there is a 
distinguished index $\0$ such that $\varphi_\0$ is constant and that 
$\lambda_\0=0$.  

Throughout, we fix a parameter $\beta>0$, which has a physical 
interpretation as the inverse temperature.  
Let $(X^{(i)}_\bk)_{\bk\in\I\setminus\{\0\},i=1,\dots,D}$ be a collection of 
i.i.d. random variables defined on a probability space 
$(\Omega,\mathcal{F},\mathbb{P})$ such that
\begin{equation*}
X^{(i)}_\bk \sim N(0,(2\beta\lambda_\bk)^{-1}).
\end{equation*}
For each $i=1,\dots,D$ define
\begin{equation} \label{u-eigen}
u^{(i)}=\sum_{\bk\in\I\setminus\0}X^{(i)}_\bk\varphi_\bk,  \\
\end{equation}
and define DGFF as  
\begin{equation} \label{DGFF-def}
\bu=(u^{(1)},\dots,u^{(D)}).
\end{equation}

Since the definition of $\bu$ in \eqref{u-eigen} may appear unmotivated, we 
remark that it is equivalent to stating that $(\bu(\x))_{\x\in S_N^d}$ is a 
collection of $\R^D$-valued centered jointly Gaussian variables with joint 
density
\begin{equation} \label{eq:prob-density}
\frac{1}{C_{N,\beta,d,D}}\exp\left(-\beta H(\bu)\right)
= \frac{1}{C_{N,\beta,d,D}}\exp\left(\beta(\bu,\Delta\bu)\right)
\end{equation}
when $\bu$ is restricted to those $\bu$ for which
\begin{equation*}
\sum_{\x\in S_N^d}\bu(\x)=0.
\end{equation*}
Here $C_{N,\beta,d,D}$ is a normalizing constant which ensures that we have a 
probability density.  We will elaborate on the equivalence of these two 
definitions in Section \ref{section-gff-density}.

We define the local time of the field $\bu$ at level 
$\z \in \mathbb R^{D}$ as 
\begin{equation} \label{l-time}
\begin{split} 
\ell_{N}(\z, \bu) &= \#\{\x \in S^d_{N} 
       : \bu(\x) \in [\z-\mathbf{1/2}, \z+\mathbf{1/2}] \} \\
&= \sum_{\x \in S_N^d} \mathbf{1}_{\{\bu(\x) \in [\z-\mathbf{1/2}, \z+\mathbf{1/2}] \}}, 
\end{split} 
\end{equation} 
where $\mathbf{\frac{1}{2}}=(1/2,\dots,1/2)\in\mathbb{R}^D$. 

Now we define a weakly self-avoiding Gaussian free field. 
Throughout, we fix a parameter $\gamma>0$.  
If $P_{N,d,D,\beta}$ denotes the 
original probability measure of $(\bu(\x))_{\x\in S_N^d}$, 
we define the probability $Q_{N,d,D,\beta,\gamma}$ as follows.  
For ease of notation, we write $E$ for the expectation with respect to 
$P_{N,d,D,\beta}$.
%clarity, we will let $E^{P_{N,d,D,\beta}},E^{Q_{N,d,D,\beta,\gamma}}$ denote the 
%expectations with respect to $P_{N,d,D,\beta}$ and $Q_{N,d,D,\beta,\gamma}$ 
%respectively.  We write $E$ for $E^{P_{N,d,D,\beta}}$.  
Let
\begin{equation} \label{z-func} 
\begin{aligned}
\mathcal{E}_{N,d,D,\gamma}
&=\exp\left(-\gamma \int_{\mathbb R^D}\ell_N(\y, \bu)^2d\y \right),\\
Z_{N,d,D,\beta,\gamma}&=E[\mathcal{E}_{N,d,D,\gamma}]
 =E^{P_{N,d,D,\beta}}[\mathcal{E}_{N,d,D,\gamma}].
\end{aligned}
\end{equation} 
Then we define for any set $A\in\mathcal{F}$,
\begin{equation} \label{eq:def-Q}
Q_{N,d,D,\beta,\gamma}(A)=\frac{1}{Z_{N,d,D,\beta,\gamma}}E
 \big[\mathcal{E}_{N,d,D,\gamma}\mathbf{1}_A\big].
\end{equation}
For ease of notation, we will usually drop the subscripts except for 
$N$ and write
\begin{equation*}
P_N=P_{N,d,D,\beta} , \quad Q_N=Q_{N,d,D,\beta,\gamma}  , \quad 
\mathcal{E}_N=\mathcal{E}_{N,d,D,\gamma} , \quad Z_N=Z_{N,d,D,\beta,\gamma}.
\end{equation*}

Finally, we define the effective radius of the field $\bu$ as   
\begin{equation*}
R_N=  \max_{\w, \z\in S^d_N}  \|\bu(\z)-\bu(\w)\|, 
\end{equation*}
where $\|\cdot\|$ denotes the Euclidian norm.

\subsection{Statement of the main result}

Note that in our main theorem below, we assume that $D=d$.  
 
\begin{theorem}\label{th:main}
Let $\bu$ be the weakly self-avoiding DGFF on $S_N^d$ taking values in 
$\mathbb{R}^d$.  There are constants $\varepsilon_0,K_0>0$ not depending on 
$\beta,\gamma,N$ such that  
\begin{itemize} 
\item[\textbf{(i)}] For $d=2$, 
\begin{align*}
\lim_{N\to\infty} Q_{N}\Big[\varepsilon_0 \Big(\frac{\gamma}{\beta+\gamma}\Big)^{1/2}  N(\log N)&^{-1/2}
\leq R_N \\
 &\leq K_{0} \Big(\frac{\beta+\gamma}{\beta}\Big)^{1/2}  N(\log N)^{3/2}\Big]  =1.
\end{align*}
\item[\textbf{(ii)}] For $d\geq 3$, 
\begin{equation*}
\lim_{N\to\infty} Q_{N}\Big[\varepsilon_0  \Big(\frac{\gamma}{\beta+\gamma}\Big)^{1/d}N
 \leq R_N\leq K_{0}\Big(\frac{\beta+\gamma}{\beta}\Big)^{1/2} N^{d/2} \Big] =1.
\end{equation*}
\end{itemize} 
\end{theorem}

\begin{remark}
Here is the reason we restrict ourselves to the case of $\bu$ taking values 
in $\mathbb{R}^D$ with $D=d$.  As 
mentioned in the introduction, there is good information about the radius 
of the self-repelling random polymer taking values in $\mathbb{R}^D$ for 
$D=1$, but not for $D=2,3,4$.  This is because we can guess that for $D=1$, 
a self-avoiding polymer has ballistic behavior, i.e. $u(x)\approx Cx$ 
roughly speaking.  In higher dimensions, it is hard to guess what shape 
the polymer might take.  There are results for $d>4$ using the lace expansion
(see den Hollander \cite{dH09}, Chapter 4),
but this method seems hard to adapt to elastic manifolds.  
However, if $D=d$, then we can guess that the self-repelling elastic 
manifold should stretch itself by dilation in all directions so that 
$\bu(\x)\approx C\x$.  This guess allows us to carry out the analysis.  
\end{remark}

\begin{remark}
Note that when $\gamma=0$, i.e. the self-repelling penalization in \eqref{eq:def-Q} is set zero, the following statement holds for any $d\geq 2$. There exists $K>0$ large enough such that,   
\begin{equation*}
\lim_{N\to\infty} P_{N,d,D,\beta}\big( R_N \leq K \beta^{1/2}\log N\big)  =1.
\end{equation*}
This result follows from Proposition \ref{lem-cov} and by repeating the same steps as in the proof of the upper bound in Theorem 2.1 of \cite{Bis2020}. Theorem \ref{th:main} therefore suggests that self-repelling elastic manifolds undergo a substantial stretching at any dimension.       
\end{remark}

 \begin{remark} 
 Theorem \ref{th:main} verifies the conjecture by Kantor, Kardar and Nelson in \cite{Kardar_87} for the case where $d=2$ and $D=2$, up to a logarithmic correction. Although in the model that was presented in \cite{Kardar_87} the DGFF is defined on the triangular lattice, the heuristics that yields their result is based on Flory's argument which also applies for the rectangular lattice.   
 \end{remark} 

\subsection{Outline of the proof}
We describe the outline for the case $d=2$ as the proof for $d\geq 3$ follows similar lines. Define the following events.  
\begin{equation} \label{events}
\begin{aligned} 
A^{(1)}_N&=\left\{R_N>K_{0}  \Big(\frac{\beta+\gamma}{\beta}\Big)^{1/2} N(\log N)^{3/2} \right\}, \\
A^{(2)}_N&=\left \{R_N<\varepsilon_0 \Big(\frac{\gamma}{\beta+\gamma}\Big)^{1/2} N(\log N)^{-1/2} \right \}. 
\end{aligned}
\end{equation}
It suffices to show that for $i=1,2$ we have
\begin{equation*}
\lim_{N\to\infty}Q_{N}\big(A^{(i)}_N\big)=0 .
\end{equation*}

From \eqref{eq:def-Q} we see that it is enough to find:
\begin{enumerate}
\item a lower bound on $Z_N$, derived in Section \ref{sec-part},
\item and an upper bound on 
$E^{P_N}\big[\mathcal{E}_N\mathbf{1}_{A^{(i)}_N}\big]$ 
for $i=1,2$, obtained in Sections \ref{sec-large-t} and \ref{sec-small}, respectively. 
\end{enumerate}
Finally, the upper bounds divided by the lower bound should vanish 
as $N\to\infty$.  

\section{ Lower Bound on the Partition Function}  \label{sec-part} 
In this section we derive the following lower bound on $ Z_{N}$. 
\begin{prop} \label{prop-z} 
Let $\beta>0$. Then there exists a constant $C>0$ not depending on $N$, $\beta$ and $\gamma$ such that  
\begin{itemize} 
\item[\textbf{(i)}] for $d=2$, 
\begin{equation*} \label{z-bnd-fin} 
\log  Z_{N} \geq  -C(\beta+\gamma) N^2\log N. 
\end{equation*}
\item [\textbf{(ii)}] for $d\geq 3$, 
 \begin{equation*} \label{z-bnd-fin-3} 
\log  Z_{N} \geq  -C(\beta+\gamma) N^d. 
\end{equation*} 
\end{itemize} 
\end{prop} 
In order to prove Proposition \ref{prop-z} we will introduce some additional definitions and auxiliary lemmas.

\subsection{The probability density and associated random walk} 
Note that \eqref{eq:def-Laplacian} implies that $\Delta_{N,d,1}$ is the 
generator of a continuous-time simple random walk $\X=(\X_t)_{t\geq 0}$ 
which is reflected at the boundary of $S^d_N$.  More precisely, if $\X$ 
attempts to leave $S^d_N$, then it stays where it is.  Later we will use $\X$ 
to obtain the lower bounds in Proposition \ref{prop-z}.

\subsection{The probability density of the GFF} \label{section-gff-density} 

Now we study the probability density formula \eqref{eq:prob-density} in more 
detail.  We focus on $D=1$, and note that $\Delta=\Delta_{N,d,1}$ is nonpositive definite.  As in 
\eqref{u-eigen}, we need to take into account the $0$ eigenvalue of $\Delta$, 
which has a constant eigenfunction $\varphi_\textbf{0}$ with 
$\textbf{0} = (0,\dots,0)\in S_N^d$.  In order to do that we let 
$\mathbf{V}^+=\mathbf{V}^+_{N,d}$ 
be the vector space with the standard basis $\{\mathbf{e}_\x: \x\in S^d_N\}$.  
We denote vectors in $\mathbf{V}^+$ as having one component $v_\x$ for each 
position $\x\in S^d_N$, so that $\bv=\sum_{\x\in S_N^d}v_\x\e_\x$.  Let
\begin{equation} \label{v-def} 
\mathbf{V} = \mathbf{V}_{N,d}
 = \left\{\bv\in\mathbf{V}^+:\sum_{\x\in S^d_N}v_\x=0\right\}.
\end{equation}

Next, we wish to put a natural measure on $\mathbf{V}$.  Clearly $\mathbf{V}$ 
is a subspace of $\mathbf{V}^+$, so we can choose an orthonormal basis 
$\{\mathbf{b}_k\}_{k=1}^{(2N+1)^d-1}$ such that each basis element is 
perpendicular to the constant vector $\sum_{\x\in S^d_N}\mathbf{e}_\x$.   
Using such a basis, we can construct Lebesgue measure on $\mathbf{V}$ in the 
usual way, to be translation invariant.  Also, any such orthonormal basis 
gives rise to the same measure, which we denote as $\mu=\mu_{N,d}$.  

Then we can define DGFF on $\mathbf{V}$ having density with respect to $\mu$ 
given by \eqref{eq:prob-density} with $D=1$. Note that if we were to use $\mathbf{V}^+$, then 
\eqref{eq:prob-density} (without the normalizing constant) would integrate 
to $\infty$.  The extension of  \eqref{eq:prob-density} to $D\geq 2$ is done 
by taking the product of the $D=1$ densities of $(u^{(i)})_{i=1}^D$, which 
are i.i.d, and then using \eqref{eq:def-Laplacian}. 
 
\subsection{The orthonormal function basis}  \label{sec-basis} We first 
specify the orthonormal basis $\{\varphi_{\bk}\}$ in \eqref{u-eigen} of 
eigenfunctions of $\Delta_{N,d,1}$ on 
$[-N, N]^d \cap \mathbb{Z}^d$ taking values in $\mathbb{R}$.  We note that 
each basis function $\varphi_{\bk}$ can be represented as a product of $d$ 
functions $\phi_{j}: \mathbb R^d \rightarrow \mathbb R$, as follows 
\begin{equation} \label{eigen}
\varphi_{\bk} (\x) = \phi_{k_1}(x_1)\dots\phi_{k_d}(x_d), 
\end{equation} 
where $\x=(x_1,\dots,x_d)$ and $\bk=(k_1,\dots,k_d)$, $-N \leq k_i \leq N$, 
and $1\leq i\leq d$.  Here $\{\phi_j\}_{j=-N}^N$ is an orthonormal basis of 
eigenfunctions of $-\Delta_{N,1,1}$, the Laplacian with Neumann boundary 
conditions on $S_N^1=\{-N,\dots,N\}$. 
Note that if $\lambda_{\bk}$ is the eigenvalue of $\varphi_{\bk}$ and 
$\lambda_k$ is the eigenvalue corresponding to $\phi_k$, then
satisfies
\begin{equation} \label{eigen-rel} 
 \lambda_{\bk} = \sum_{i=1}^d \lambda_{k_i}. 
\end{equation} 

We can explicitly define these eigenfunctions and eigenvalues as follows.  
Let
\begin{equation*}
\phi_{0}(x)=(2N+1)^{-1/2}, 
\end{equation*}
and for $k=1,\dots,N$ denote
\begin{equation} \label{phi-sin} 
\phi_{k}(n)=\frac{1}{a_{k,N}}\sin\left(\frac{(2k-1)\pi}{2N+1}n\right),
\end{equation}
where $a_{k,N}$ is a normalizing constant so that
\begin{equation*}
\sum_{n=-N}^{N}\phi^2_{k}(n)=1.
\end{equation*}
A calculation (perhaps using Wolfram Alpha) yields
\begin{equation*} 
a_{k,N} = \Big(  \frac{1}{2} \csc\left(\frac{( 2 k-1) \pi }{( 2 N+1)}\right) \sin(2 k \pi )+N+\frac{1}{2}   \Big)^{1/2} 
= \Big(  N+\frac{1}{2}   \Big)^{1/2}, 
\end{equation*}
since $k$ is an integer.  

For $k=1,..,N$ we further define 
 \begin{equation*}
  \phi_{-k}(n)=\frac{1}{b_{k,N}}\cos\left(\frac{2k\pi}{2N+1} n \right),
 \end{equation*}
where as in the previous case, $b_{k,N}$ is a normalizing constant such that
\begin{equation*}
\sum_{n=-N}^{N}\phi^2_{-k}(n)=1.
\end{equation*}
As before, we can compute
\begin{equation*}
b_{k,N} = \left( \frac{1}{2} \csc\left(\frac{2 k \pi}{2 N+1}\right) \sin(2 k \pi) +N +\frac{1}{2} \right)^{1/2}= \left(  N +\frac{1}{2} \right)^{1/2}, 
\end{equation*}
since $k$ is an integer. 

Our basis comprises all such combinations as in \eqref{eigen}, excluding the constant eigenfunction 
\begin{equation*}  
 \varphi_{(0,\dots,0)} (\x) = \phi_0(x_1)\dots\phi_{0}(x_d). 
\end{equation*} 
We denote by $N(d)$ the number of function in our basis,
\begin{equation*}
 N(d) = |S_N^d|-1= (2N+1)^d-1. 
\end{equation*}

\subsection{Incorporating drift} 
 
From \eqref{eq:prob-density} it follows that $C_{N,\beta,d,D}$ is given by 
\begin{equation} \label{z-n} 
\begin{aligned} 
C_{N,\beta,d,D} &= \int_{\mathbb{R}^{N(d)} }
  \exp\left(-\sum_{i=1}^{d}\sum_{k=1}^{N(d)}
   \frac{(x^{(i)}_k)^2}{2(2\beta\lambda_k)^{-1}}\right) 
   \prod_{i=1}^{d} \prod_{k=1}^{N(d)}dx^{(i)}_k\\
&=  \frac{1}{(2\beta)^{dN(d)/2}}\prod_{k=1}^{N(d)}\frac{1}{\lambda_k^{d/2}}. 
\end{aligned}  
\end{equation}
 
Next we incorporate a linear drift into each of the components $u^{(i)}$ of 
$\bu$, calling the resulting component $u^{(i)}_a$.  
\begin{equation*} 
u^{(i)}_a(\x)
  =\sum_{k=1}^{N(d)}X^{(i)}_\bk\varphi_\bk(\x) + ax_i, \quad i=1,\dots,d
\end{equation*}
where $a>0$ is a constant to be determined later. 

Using \eqref{u-eigen} and \eqref{sec-basis}, we get
\begin{equation} \label{u-drf}
u_a^{(i)}(\x) =\sum_{(k_1,\dots,k_d) \in S_N^d \setminus \{\mathbf{0}\}  } X^{(i)}_{k_1,\dots,k_d}\phi_{k_1}(x_1)\dots.\phi_{k_d}(x_d) +ax_i. 
\end{equation}
where $\mathbf{0} = (0,\dots,0)\in\mathbb{Z}^d$. 

Regarding $x_i$ as a function on $S_N^1=\{-N,\dots,N\}$ and expanding it 
in terms of our eigenfunctions, we see that there are coefficients 
$\alpha_j^{(i)}$ such that 
\begin{equation} \label{u-drf2}
x_i=  (\phi_{0})^{d-1} \sum_{j\in S_N^1\setminus \{0\}}  \phi_{j}(x_i) \alpha_{j}^{(i)}, 
\end{equation}
where we have included $(\phi_0)^{(1-d)}$ for convenience in later 
calculations.  Recall that 
\begin{equation*}
\phi_0=\phi_{0}(x)=(2N+1)^{-1/2}. 
\end{equation*}
In \eqref{u-drf2}, we do not include $j=0$ because $x_i$ is orthogonal to 
the constant function $\phi_0$.  

Next we find $\alpha_j^{(i)}$ in \eqref{u-drf2}. Since 
$\alpha_j^{(i)}$ are used to expand the function $f(x)=x$ for each 
coordinate $i$, we can omit the superscript $i$ and write just $\alpha_j$ in 
what follows.  Since $\{ \phi_j \}_{j=-N}^N$ forms an orthonormal basis, we 
get  
\begin{equation}  \label{f-coef} 
\alpha_{j} =\phi_0^{(1-d)} \sum_{n=-N}^{N}n \phi_j(n), \quad j\not =0, \quad \textrm{and } \quad \alpha_{0}=0. 
\end{equation} 

From \eqref{u-drf} and \eqref{u-drf2} we have  
 \begin{equation*}
 \begin{aligned} 
u_a^{(i)}(\x)& =\sum_{(k_1,\dots,k_d) \in S_N^d \setminus \{\textbf{0}\}   } X^{(i)}_{k_1,\dots,k_d}\phi_{k_1}(x_1)\dots.\phi_{k_d}(x_d) \\
&\quad +a\phi_{0}^{d-1} \sum_{j \in S_N^1 \setminus \{0\}}  \phi_{j}(x_i) \alpha_{j}. 
\end{aligned} 
\end{equation*}
We can represent $u_a^{(i)}$ as follows: 
\begin{equation} \label{u-drift} 
 \begin{aligned} 
u_a^{(i)}(x)& =\sum_{(k_1,\dots,k_d) \in S_N^d \setminus \{j\e_i: \, j\in S_N^1\}} X^{(i)}_{k_1,\dots,k_d}\phi_{k_1}(x_1)\cdots\phi_{k_d}(x_d) \\
&\quad +\phi_{0}^{d-1} \sum_{j \in S_N^1\setminus \{0\}}  (X_{j\e_i}^{(i)} +a \alpha_{j}^{(i)}) \phi_{j}(x_i). 
\end{aligned} 
\end{equation}
For $i=1,\dots,d$ let $\x^{(i)} \in \mathbf{V}$ and define 

\begin{equation*}
F(\bold x^{(1)},\dots,\bold x^{(d)}) = \sum_{i=1}^{d}\sum_{(k_1,\dots,k_d)\in S_N^d\setminus\{\mathbf{0}\}}\frac{(x^{(i)}_{k_1,\dots,k_d})^2}{2(2\beta\lambda_{k_1,\dots,k_d})^{-1}}.
\end{equation*}
We rewrite $Z_{N}$ in \eqref{z-func} as follows.  We should emphasize that 
the local time $\ell_N$ is random and hence a function of the random 
variables $(X^{(i)}_{\bk})$, hence we can write
\begin{equation*}
g_N(\y,(X_{\bk}^{(i)})) 
:= \ell_N\left( \y, \bu\right)
\end{equation*}
where for readability we have omitted the specification that $i=1,\dots,d$ 
and $\bk\in S_N^d\setminus\{\mathbf{0}\}$. We get that,  
\begin{equation} \label{z-hat-mod} 
\begin{aligned}
Z_{N}
&=  \int_{\mathbb{R}^{N(d)}} \exp\left(-F(\bold x^{(1)},\dots,\bold x^{(d)})  -\gamma \int_{\mathbb{R}^d} g_N^{2}\big(\y,(\x^{(i)})\big)d\y\right) \\
&\quad \times \prod_{i=1}^d\prod_{(k_1,\dots,k_d)\in S_N^{d}\setminus\{\textbf{0}\}}dx_{k_1,\dots,k_d}^{(i)}  . 
\end{aligned}
\end{equation}  
In order to find the Radon-Nikodym derivative that allows the drift addition in \eqref{u-drift} we note that, 
\begin{equation} \label{sums-sums} 
\begin{aligned} 
&\sum_{i=1}^{d}\sum_{(k_1,\dots,k_d)\in S_N^{d}\setminus\{\textbf{0}\}}\frac{(x^{(i)}_{k_1,\dots,k_d})^2}{2(2\beta\lambda_{k_1,\dots,k_d})^{-1}}  \\ 
&= \sum_{i=1}^{d} \bigg( \sum_{ (k_1,\dots,k_d) \in S_N^d \setminus \{j\textbf{e}_{i}: \, j\in S_N^1 \}}\frac{(x^{(i)}_{k_1,\dots,k_d})^2}{2(2\beta\lambda_{k_1,\dots,k_d})^{-1}}  \\
&\qquad  + \sum_{j\in S_N^1 \setminus \{0\}}\frac{(x^{(i)}_{ j\e_{i}} +a\alpha_{j})^2}{2(2\beta\lambda_{j\e_i})^{-1}} -\sum_{j\in S_N^1 \setminus \{0\}}  \frac{2a\alpha_{j}x^{(i)}_{j\e_i}+(a\alpha_{j})^{2}}{2(2\beta\lambda_{j\e_i})^{-1}}  \bigg).\end{aligned} 
\end{equation}  

We therefore define $\hat P^{(a)}$ (resp. $\hat E^{(a)}$) be the measure 
(expectation) under which $u$ is shifted as in \eqref{u-drift}. Then 
\eqref{z-hat-mod} and \eqref{sums-sums} imply
\begin{equation} \label{hat-P-meas} 
\frac{d\hat P^{(a)}}{dP}= \exp \bigg(-\sum_{i=1}^{d}\sum_{j\in  S_N^1 \setminus \{0\}}  \frac{2a\alpha_{j}x^{(i)}_{j\e_i}+(a\alpha_{j})^{2}}{2(2\beta\lambda_{ j\e_i})^{-1}} \bigg) .
\end{equation}
 
We can therefore rewrite  $Z_N$ in \eqref{z-hat-mod} as follows,  
\begin{equation} 
\begin{split}
Z_{N}
&= \hat E^{(a)}\Bigg[  \exp\bigg( \sum_{i=1}^{d}\sum_{j\in S_N^1 \setminus \{0\}}  \frac{2a\alpha_{j}X^{(i)}_{j\e_i}+(a\alpha_{j})^{2}}{2(2\beta\lambda_{j\e_i})^{-1}} \\
& \hspace{4cm} -\gamma \int_{\mathbb{R}^d} \ell_N^{2}(\y, \bu)d\y\bigg) \Bigg]. 
\end{split}
\end{equation}   
 
We define  
\begin{equation} \label{y-rv} 
Y_{j\e_i}^{(i)} =    \frac{2a\alpha_{j}X^{(i)}_{j\e_i}+(a\alpha_{j})^{2}}{2(2\beta\lambda_{j\e_i})^{-1}}.
\end{equation}

Using Jensen's inequality, we get that 
\begin{equation} \label{log-z} 
\begin{aligned}
\log   Z_{N}  
&\geq \hat E^{(a)}\left[-\gamma \int_{\mathbb{R}^d} \ell_N^{2}(\y,\bu)d\y\right] 
  - \hat E^{(a)}\left[- \sum_{i=1}^d\sum_{j\in S_N^1\setminus 
\{0\}}Y_{j\e_i}^{(i)}    \right]  \\
&=:-(I_{1} +I_{2}).
\end{aligned}
\end{equation}
The following proposition gives some essential bounds on $I_i$, $i=1,2$. 
\begin{prop} \label{prop-bnd-I} 
Let $\beta,\gamma >0$. Then there exists a constant $C>0$ not depending on 
$N,\beta,\gamma$ such that

\begin{itemize} 
\item [\textbf{(i)}]  for $d=2$,
\begin{equation*}
I_1   \leq  C \gamma  \big((a^{-2}(\log N)^2)\vee 1\big) N^2 ,
\end{equation*}
\item [\textbf{(ii)}] for  $d\geq 3$,
\begin{equation*}
I_1 \leq  C \gamma  (a^{-2 } \vee 1)N^d,
\end{equation*}
\item  [\textbf{(iii)}]  for any $d\geq 2$, 
\begin{equation*}
I_2  \leq C\beta  a^2 N^d.
\end{equation*}
\end{itemize} 
\end{prop} 
The proof of Proposition \eqref{prop-bnd-I}(i) and (ii) is postponed to Section \ref{sec-pf-z-1}. The proof of Proposition \eqref{prop-bnd-I}(iii) is given in Section \ref{sec-pf-z-2}. 
\subsection{Proof of Proposition \ref{prop-z}} 
\begin{proof} [Proof of Proposition \ref{prop-z}] 
From \eqref{log-z} and Proposition \ref{prop-bnd-I}(i) and (iii) it follows that for $d=2$,  
\begin{equation} \label{z-lb} 
\log \hat Z_{N} \geq -( I_1+I_2)\geq  -C \left[\gamma \big((a^{-2}(\log N)^2) \vee 1 \big) N^{2} +\beta N^2  a^2\right].
\end{equation} 
Taking $a^2=\log N$ we have 
\begin{equation*}
 \log \hat Z_{N} \geq  -C(\beta + \gamma) N^2 \log N . 
\end{equation*}

The proof for $d \geq 3$ follows the same lines with the only modification that we are using Proposition \ref{prop-bnd-I}(ii) and choosing $a=1$.

\end{proof}

\section{Proof of Proposition \eqref{prop-bnd-I}(i) and (ii)}  \label{sec-pf-z-1} 
\begin{proof} [Proof of Proposition \ref{prop-bnd-I}(i) and (ii)] 
 We can write 
\begin{equation} \label{i-1-eq} 
\begin{aligned} 
\tilde I_{1}&:=\hat{E}^{(a)}\left[ \int_{\mathbb R^d} \ell_N(\y,\bu)^2d\y \right] \\
&= \hat{E}^{(a)}\left[ \int_{\mathbb{R}^d} \Big( \sum_{\z\in S^d_N}\mathbf{1}_{[\y-\mathbf{1/2},\y+\mathbf{1/2}]}(\bu(\z)) \Big)^2d\y  \right]  \\ 
&=  \sum_{\z\in S^d_N} \hat{E}^{(a)}  \left[\int_{\mathbb R^d} \mathbf{1}_{[\y-\mathbf{1/2},\y+\mathbf{1/2}]}(\bu(\z))d\y \right] \\ 
&\quad + \sum_{\z,\w\in S^d_N, \, \z\not =\w}\hat{E}^{(a)}  \left[\int_{\mathbb R^d} \mathbf{1}_{\bu(\z) ,\bu(\w) \in [\y-\mathbf{1/2},\y+\mathbf{1/2}]} d\y \right]  \\
&=  (2N+1)^d 
 +  \sum_{\z,\w\in S^d_N,\,\z\not=\w}\int_{|\y|\le1}
          \hat p_{\z,\w}(\y)d\y,
\end{aligned} 
\end{equation}
where $ \hat p_{\z,\w}$ is the density of $\bu(\z)-\bu(\w)$ under $\hat P^{(a)}$. 

We will need the following proposition which will be proved in Section \ref{sec-pf-prop-var}.  
\begin{prop} \label{lem-cov} 
There exist constants $C_1,C_2>0$ such that, 
\begin{itemize} 
\item[\textbf{(i)}] for $d=2$, for all $\w,\z\in S_N^2$ with 
$\w\ne\z$, and for $i=1,2$ we have
\begin{equation*} 
C_1 \beta^{-1} \leq \textrm{Var}\left(u^{(i)}(\z)-u^{(i)}(\w)\right) \leq C_2 \beta^{-1}  (\log N)^2,
\end{equation*} 
\item[\textbf{(ii)}] for $d \geq 3$, for all $\w,\z\in S_N^d$ with 
$\w\ne\z$, and for $i=1,\dots,d$ we have
\begin{equation*} 
C_1 \beta^{-1}  \leq \textrm{Var}\left(u^{(i)}(\z) - u^{(i)}(\w)\right) \leq C_2 \beta^{-1 }  .
\end{equation*} 
\end{itemize} 
\end{prop} 
Note that from \eqref{u-drf} we have  
\begin{equation*}
\hat E^{(a)}[ u^{(i)}(\z) - u^{(i)}(\w) ]   = a(z_i-w_i), \quad 
\text{ for }i=1,\dots,d.   
\end{equation*}
Since $(u^{(i)})_{i=1,\dots,d}$ are independent, we have for any $\bold y \in \mathbb{R}^d$
\begin{equation} \label{prod}  
 \hat p_{\z,\w}(\bold y) := \prod_{i=1}^d\hat p^{(i)}_{\z,\w}( y_i)
\end{equation} 
and therefore
\begin{equation} \label{prod-1}  
 \int_{\|\y\|\le1}\hat{p}_{\z,\w}(\y)d\y
\le  \prod_{i=1}^d\int_{-1}^{1}\hat{p}^{(i)}_{\z,\w}(y_i)dy_i, 
\end{equation} 
where $\hat{p}^{(i)}_{\z,\w}$ is the density of $u^{(i)}(\z)-u^{(i)}(\w)$ under $\hat P^{(a)}$.

We distinguish between the following two cases.

 \textbf{Case 1:} $d=2$. Then from Proposition \ref{lem-cov}(i) we have 
\begin{equation} \label{p-bnd}  
 \hat p^{(i)}_{\z,\w}( y_i) \leq  \frac{1}{\sqrt{2\pi C_1 \beta^{-1}}} \exp\Big( -\frac{a^2(y_i-(z_i-w_i))^2}{2\beta^{-1} C_2(\log N)^2} \Big).
 \end{equation}
 From \eqref{prod-1} and \eqref{p-bnd} we therefore get 
 \begin{equation} \label{sum-p-bnd}
\begin{aligned} 
&\sum_{\z,\w\in S_N^2, \, \z\not =\w} \int_{\|\y\|\leq 1}  \hat p_{\z,\w}(\y) d\y \\
& \leq\sum_{\z,\w\in S_N^2, \, \z\not =\w} \int_{\y \in [-1,1]^2}  \frac{  1}{2\pi C_1 \beta^{-1}}\exp\Big( -\frac{a^2\|\z-\w - \y\|^2}{2C_2\beta^{-1}(\log N)^2} \Big)d\y \\
& \leq  \int_{\y \in [-1,1]^2} J(\y) d\y. 
\end{aligned}
 \end{equation}

where 
$$
J(\y):=\sum_{\z,\w\in S_N^2}   \frac{  1}{2\pi C_1\beta^{-1}}\exp\Big( -\frac{a^2\|\z-\w - \y\|^2}{2C_2\beta^{-1} (\log N)^2} \Big).
$$
We will use the following lemma, which will be proved in the end of this section. 
\begin{lemma}\label{lem-cont-sum} 
Let $\kappa>0$. Then for all $y\in [-1,1]$ and $w \in S^1_N$ we have 
$$
 \sum_{z\in S^1_N}    \exp\Big( -\kappa (z-w -y)^2 \Big) 
\leq 3 +  \int_{-\infty}^{\infty} \exp\Big( -\kappa (z-w - y)^2 \Big)dy. 
$$
\end{lemma}  
 Using Lemma \ref{lem-cont-sum} and integrating over the Gaussian density gives,    
\begin{equation} \label{j-bnd} 
\begin{aligned} 
J( \y)   
&= \frac{  1}{2\pi C_1\beta^{-1} }\left(\sum_{w_1\in S_N^1} \sum_{z_1\in S_N^1}   \exp\Big( -\frac{a^2(z_1-w_1 - y_1)^2}{2C_2\beta^{-1} (\log N)^2} \Big)\right)\\
&\quad \times\left(\sum_{w_2\in S_N^1} \sum_{z_2\in S_N^1}  \exp\Big( -\frac{a^2(z_2-w_2 - y_2)^2}{2C_2\beta^{-1}(\log N)^2} \Big)\right)\\
&\leq Ca^{-2}(\log N)^2 \\
&\quad \times \sum_{w_1\in S_N^1}\Bigg(3+ \frac{1}{\sqrt{ 2\pi  C_2\beta^{-1} a^{-2}   (\log N)^2}}  \\
&\hspace{3cm}\times\int_{-\infty}^\infty\exp\Big( -\frac{(z_1-w_1-y_1)^2}{2  C_2\beta^{-1} a^{-2}(\log N)^2} \Big) dz_1 \Bigg)  \\
&\quad \times\sum_{w_2 \in S_N^1}\Bigg(3+  \frac{1}{\sqrt{ 2\pi  C_2\beta^{-1}  a^{-2}(\log N)^2}} \\
&\hspace{3cm}\times\int_{-\infty}^\infty \exp\Big( -\frac{(z_2-w_2-y_1)^2}{2   C_2\beta^{-1}a^{-2}  (\log N)^2} \Big)dz_2\Bigg) \\ 
&\leq C a^{-2} (\log N)^2 \sum_{w_1 \in S_N^1} \sum_{w_2 \in S_N^1}  16\\
&\leq Ca^{-2}  N^2(\log N)^2. 
\end{aligned}
\end{equation}
The number 16 in the next to last line above comes from evaluating the 
standard Gaussian integrals in the previous two lines.  

Substituting \eqref{j-bnd} into \eqref{sum-p-bnd} gives 
\[
\sum_{\z,\w\in S_N^2, \, \z\not =\w}  \int_{\|\y\|\leq 1}   \hat p_{\z,\w}(\y) d\y \leq Ca^{-2}  N^2(\log N)^2. 
\]
Together with \eqref{i-1-eq} we get that 
\begin{equation} \label{i-1} 
\begin{aligned} 
\tilde I_{1} 
&\leq  (2N+1)^2 + C a^{-2 }    N^{2}(\log N)^2 \\ 
&\leq \tilde C (a^{-2}(\log N)^2\vee 1)  N^{2}. 
\end{aligned}
\end{equation}
 \textbf{Case 2:} $d \geq 3$. From Proposition \ref{lem-cov}(ii) we have

\begin{equation} \label{p-bnd-3}  
 \hat p^{(i)}_{\z,\w}( y_i) \leq  \frac{1}{\sqrt{2\pi C_1\beta^{-1}}} \exp\Big( -\frac{a^2(z_i-w_i-y_i)^2}{2C_2 \beta^{-1}} \Big).
 \end{equation}
 From \eqref{prod-1} and \eqref{p-bnd-3} we therefore get 
\begin{align*}  
&\sum_{\z,\w\in S_N^d, \, \z\not =\w}  \int_{\| \y\| \leq1}  \hat p_{\z,\w}(    \y) \\
& \leq C\sum_{\z,\w\in S_N^d, \, \z\not =\w}  \int_{ \y \in [-1,1]^d}      \frac{1}{(2\pi C_1\beta^{-1})^{d/2}} \exp\Big( -\frac{a^2\|\z-\w-\y\|^2}{2C_2\beta^{-1}} \Big)d\y.
\end{align*}
The rest of the proof is similar to Case 1, with the only difference that we do not have the dependence in $\log N$ in the Gaussian density. This leads to, 
\begin{equation*}
\tilde I_{1} 
 \leq \tilde C (a^{-2}\vee 1 ) N^{d}. 
\end{equation*}
Since from \eqref{log-z} we have that 
\begin{equation*}
I_1 = \gamma  \tilde I_{1},  
\end{equation*}
this completes the proof of Proposition \ref{prop-bnd-I} parts (i) and (ii).
\end{proof} 

\begin{proof}[Proof of Lemma \ref{lem-cont-sum}] 
Using the fact that  $f(z):=\exp\Big( -\kappa (z -w - y)^2 \Big)$ is monotone decreasing for $z>w+y$ and since $y\in [-1,1]$ we get,  
\begin{equation} \label{comp1}  
\begin{aligned} 
 \sum_{z\geq w+2} \exp\Big( -\kappa (z -w - y)^2 \Big) 
&\leq \int_{w +1}^{\infty}  \exp\Big( -\kappa (z -w- y)^2 \Big)dz.
\end{aligned} 
 \end{equation} 
Similarly using the fact that $f(z)$ is monotone increasing for $z<w+y$ we get 
\begin{equation} \label{comp2}  
\begin{aligned} 
 \sum_{z \leq w-2} \exp\Big( -\kappa (z -w - y)^2 \Big) 
&\leq \int_{-\infty}^{w -1}  \exp\Big( -\kappa (z -w- y)^2 \Big)dz.
\end{aligned} 
\end{equation} 
Finally, we have, 
\begin{equation} \label{comp3}  
\begin{aligned} 
 \sum_{z= w-1}^{w+1} \exp\Big( -\kappa (z -w - y )^2 \Big) 
&\leq 3.
\end{aligned} 
\end{equation} 
From \eqref{comp1}--\eqref{comp3} we get 
$$
\sum_{z \in S_N^1}\exp\Big( -\kappa (z-w - y)^2 \Big) 
\leq 3 + \int_{-\infty}^{\infty}  \exp\Big( -\kappa (z -w - y)^2 \Big)dz. 
$$
\end{proof}
\section{Proof of Proposition \eqref{prop-bnd-I} (iii) } \label{sec-pf-z-2}
\begin{proof}[ Proof of Proposition \eqref{prop-bnd-I} (iii) ] 
Recall that $I_2$ was defined in \eqref{log-z}, 
\begin{equation} \label{l2-dec} 
\begin{aligned} 
I_2 &= -\sum_{i=1}^d\sum_{\ell \in S_N^1\setminus \{0\}} \hat E^{(a)}\left[ Y_{\ell\e_i}^{(i)}    \right], 
\end{aligned}
\end{equation}
where $Y_{\ell\e_i}^{(i)}$ was defined in \eqref{y-rv}. 
We further introduce the following notation: 
\begin{equation*}
z_{k_1,\dots,k_d}^{(i)} =\frac{(x^{(i)}_{k_1,\dots,k_d})^2}{2(2\beta\lambda_{k_1,\dots,k_d})^{-1}}  , \quad w^{(i)}_{j\e_i} = \frac{(x^{(i)}_{j\e_i} +a\alpha_{j})^2}{2(2\beta\lambda_{j\e_i})^{-1}}, 
\end{equation*}
and 
\begin{equation*}
y_{\ell \e_i}^{(i)} =    \frac{2a\alpha_{\ell}x^{(i)}_{\ell\e_i}+(a\alpha_{\ell})^{2}}{2(2\beta\lambda_{\ell\e_i})^{-1}}.
\end{equation*}
Then from \eqref{sums-sums} and \eqref{hat-P-meas} we have 
\begin{equation}  \label{y-calc} 
\begin{aligned}
& \hat E^{(a)}\left[  Y_{\ell \e_i}^{(i)} \right]  \\
&= \frac{1}{C_{N,\beta,d,D} } \int y_{\ell\e_i}^{(i)}  \exp\bigg( -\sum_{m=1}^d \Big( \sum_{ (k_1,\dots,k_d) \in S_N^d \setminus \{j\e_m: \, j\in S_N^1\}}z_{k_1,\dots,k_d}^{(m)}  \\
& \qquad    + \sum_{l\in S_N^1 \setminus \{0\}}w^{(i)}_{l\e_m} \Big)\bigg)\prod_{r=1}^{d} \prod_{k=1}^{N(d)}dx^{(r)}_k,
\end{aligned}
\end{equation} 
where $C_{N,\beta,d,D}$ was defined in \eqref{z-n}.  

Since the expected value in \eqref{y-calc} is symmetric with respect to $i$, we can use $i=1$ in what follows in order to ease the notation.  
\begin{equation}  \label{y-1-exp} 
\begin{split}
& \hat{E}^{(a)}\left[  Y_{\ell \e_1}^{(1)} \right]
  = \frac{1}{C_{N,\beta,d,D} } \int y_{\ell \e_1}^{(1)} \exp(-w^{(1)} _{\ell\e_1} )   \\
&\quad  \times \int \exp\bigg(- \Big(\sum_{i=1}^d  \sum_{ (k_1,\dots,k_d) \in S_N^d \setminus \{j\e_i: \, j\in S_N^1 \}}z_{k_1,\dots,k_d}^{(i)}  \\
& \quad    + \sum_{i=2}^d\sum_{j\in \{-N,\dots,N\}\setminus \{0\}}w^{(i)}_{j\e_i} +\sum_{j\in S_N^1 \setminus \{0,\ell\}}w^{(1)}_{j\e_1}   \Big)\bigg)\prod_{r=1}^{d} \prod_{k=1}^{N(d)}dx^{(r)}_k.
\end{split}
\end{equation} 
We notice that we have three types of integrals above, which can be evaluated as follows. 
We have $d((2N+1 )^d - (2N+1))$ integrals of the form 
\begin{equation*} 
\begin{aligned} 
\int_{\mathbb{R}} \exp \big(-z_{k_1,\dots,k_d}^{(i)} \big) dx^{(i)}_{k_1,\dots,k_d}&=
\int_{\mathbb{R}} \exp\left(- \frac{(x^{(i)}_{k_1,\dots,k_d})^2}{2(2\beta\lambda_{k_1,\dots,k_d})^{-1}}  \right) dx^{(i)}_{k_1,\dots,k_d} \\
& =  \sqrt{2\pi}  (2\beta\lambda_{k_1,\dots,k_d})^{-1/2}. 
\end{aligned}
 \end{equation*} 
We have $2N(d-1)+2N-1$ integrals of the form 
\begin{equation*} 
\begin{aligned}  
\int_{\mathbb{R}} \exp\big(-w^{(i)}_{j\e_i} \big)dx^{(i)}_{j\e_i}  &=
\int_{\mathbb{R}} \exp\left(- \frac{(x^{(i)}_{\textbf 0_{i}(j)} +a\alpha_{j })^2}{2(2\beta\lambda_{ j\e_i})^{-1}} \right) dx^{(i)}_{j\e_i} \\
& =  \sqrt{2\pi}  (2\beta\lambda_{j\e_i})^{-1/2},
\end{aligned}
\end{equation*} 
and one integral as follows 
\begin{equation*} 
\begin{aligned}  
& \int y_{\ell\e_1} \exp(-w^{(1)} _{\ell\e_1} )dx_{\ell\e_1}^{(1)} \\ &=\int_{\mathbb{R}}  \frac{2a\alpha_{\ell}x^{(1)}_{j\e_1}+(a\alpha_{\ell})^{2}}{2(2\beta\lambda_{\ell\e_1})^{-1}} \exp\left(-\frac{(x^{(1)}_{\ell\e_1} +a\alpha_{\ell})^2}{2(2\beta\lambda_{\ell\e_1})^{-1}} \right)dx_{\ell\e_1}^{(1)} \\ 
&= - \sqrt{2\pi} \frac{(a\alpha_{\ell})^2}{2(2\beta\lambda_{\ell\e_1})^{-1/2}}.
\end{aligned}
\end{equation*} 
Plugging in all the above integrals into \eqref{y-1-exp} gives 
\begin{equation*} 
\begin{aligned} 
& \hat E^{(a)}\left[  Y_{{\ell} \e_1}^{(1)} \right]\\
 &= -  \frac{1}{C_{N,\beta,d,D}} \sqrt{2\pi} \frac{(a\alpha_{{\ell}})^2}{2(2\beta\lambda_{{\ell}\e_1})^{-1/2}} \\
 &\qquad \times  \prod_{i=1}^d \prod_{ (k_1,\dots,k_d) \in S_N^d \setminus \{j\e_i: \, j\in S_N^1 \}}\sqrt{2\pi}  (2\beta\lambda_{k_1,\dots,k_d})^{-1/2} \\
 &\qquad \times \prod_{i=2}^d \prod_{j\in  S_N^1 \setminus \{0\}}\sqrt{2\pi}  (2\beta\lambda_{{ j\e_i}})^{-1/2} \prod_{j\in S_N^1 \setminus \{0,{\ell}\}} \sqrt{2\pi}  (2\beta\lambda_{ j\e_1})^{-1/2} \\
 &=    -  \frac{1}{C_{N,\beta,d,D}} \frac{(2\pi)^{((2N+1)^{d}-1)d/2}(a\alpha_{{\ell}})^2\lambda_{ {\ell}\e_1}}{2(2\beta)^{((2N+1)^{d}-1)d/2-1}}  \prod_{i=1}^d  \prod_{(k_1,\dots,k_d) \in S_N^d\setminus\{\textbf 0\}}\frac{1}{ \lambda_{k_1,\dots,k_d}^{1/2}}.
\end{aligned}
\end{equation*}
 
Together with \eqref{z-n} we get 
\begin{equation} \label{ir1} 
\begin{aligned}  
  \hat E^{(a)}\left[  Y_{\ell\e_1}^{(1)} \right] &=- \beta(a\alpha_{\ell})^2\lambda_{\ell\e_1}. 
\end{aligned}
\end{equation}

 Plugging  \eqref{ir1} into \eqref{l2-dec} gives 
\begin{equation} \label{i2-fin} 
\begin{aligned} 
I_2&= \beta \sum_{i=1}^d \sum_{ \ell \in S_N^1 \setminus \{0\}}(a\alpha_{\ell})^2\lambda_{\ell \e_i}  \\
  &=d\beta a^2  \sum_{\ell \in S_N^1 \setminus \{0\}}\alpha_{\ell}^2\lambda_{\ell \e_1},
\end{aligned}
\end{equation}
where we used the fact that $\lambda_{\ell\e_i} = \lambda_{\ell\e_1}$, by the 
symmetry of the eigenvalues (see \eqref{eigen-rel}).

In order to complete the proof we introduce the following lemma, which will 
be proved at the end of this section. 
\begin{lemma}\label{lemma-coef}  There exists a constant $C>0$ not depending on $N$ and $\beta$ such that, 
\begin{equation*}
  \sum_{j \in S_N^1 \setminus \{0\}}\alpha_{j}^2\lambda_{j\e_1} \leq CN^d. 
\end{equation*}
\end{lemma} 
 From Lemma \ref{lemma-coef} and \eqref{i2-fin} we conclude that 
\begin{equation}  
I_2 \leq  C\beta a^2  N^2, 
\end{equation}
 which completes the proof of Proposition \ref{prop-bnd-I} part (iii).
\end{proof}
 
\begin{proof}[ Proof of Lemma \ref{lemma-coef}] 
From \eqref{f-coef} we have 
\begin{equation*}  
n= \phi_{0}^{d-1} \sum_{j \in S_N^1\setminus \{0\}}  \phi_j(n) \alpha_{j}, 
\end{equation*}
where we recall that
\begin{equation} \label{phi0} 
\phi_0=\phi_{0}(x)=(2N+1)^{-1/2}. 
\end{equation}
Since $ \{\phi_j\}_{-N}^N$ are orthonormal we get 
\begin{equation*}   
\alpha_{j} =  \sum_{n=-N}^{N}\frac{n}{\phi_0^{d-1}} \phi_j(n), \quad j\not =0, \quad \textrm{and } \quad \alpha_{0}=0. 
\end{equation*} 

We recall the eigenfunctions from Section \ref{sec-basis}. For the cosine eigenfunctions we have 
\begin{equation*}
\sum_{-N}^Nn \cos\left(\frac{2k\pi}{2N+1} n \right)  = 0, 
\end{equation*}
and therefore for all $k=1,..,N$, 
\begin{equation} \label{0-coef} 
 \alpha_{-k} =0.
\end{equation}
So together with \eqref{phi-sin} we have 
\begin{equation*}   
\begin{aligned} 
\alpha_{k} &=  \sum_{n=-N}^{N}\frac{n}{\phi_0^{d-1}} \phi_{k}(n) \\
&=(2N+1)^{(d-1)/2}\sum_{n=-N}^{N} n \phi_{k}(n) \\
&=(2N+1)^{(d-1)/2} \frac{1}{(N+1/2)^{1/2}}\sum_{n=-N}^{N} n\sin\left(\frac{(2k-1)\pi}{2N+1}n \right) \\
&  = \sqrt{2} (2N+1)^{(d-2)/2}\sum_{n=-N}^{N} n\sin\left(\frac{(2k-1)\pi}{2N+1}n \right). 
\end{aligned}  
\end{equation*}
Let
\begin{equation*}
Y := \frac{(2k-1)\pi}{2N+1}.
\end{equation*}
Then we have,
\begin{equation*}
\begin{split}
\frac{\alpha_k}{\sqrt{2}(2N+1)^{(d-2)/2}}&=\sum_{n=-N}^{N}n\sin\left(Y n\right) \\
&= \frac{1}{2}\csc^2\left(\frac{Y}{2}\right)
   \Big[(N+1)\sin(NY)-N\sin((N+1)Y)\Big]  \\
&= \frac{(N+1)\sin(NY)-N\sin((N+1)Y)}
        {2\sin^2\left(\frac{Y}{2}\right)}  \\
&= \frac{\sin(NY)+2N\sin\left(-\frac{Y}{2}\right)
     \cos\left(\left(N+\frac{1}{2}\right)Y\right)}
       {2\sin^2\left(\frac{Y}{2}\right)}.
\end{split}
\end{equation*}
Note that 
\begin{equation*}
\left(N+\frac{1}{2}\right)Y = \frac{1}{2}(2N+1)Y 
   = \left(k-\frac{1}{2}\right)\pi, 
\end{equation*}
and so for any integer $k$, 
\begin{equation*}
\cos\left(\left(N+\frac{1}{2}\right)Y\right)
= \cos\left(\left(k-\frac{1}{2}\right)\pi\right)
= 0.
\end{equation*}
Therefore
\begin{equation*}
\frac{\alpha_k}{\sqrt{2}(2N+1)^{(d-2)/2}} 
= \frac{\sin(NY)}{2\sin^2\left(\frac{Y}{2}\right)}. 
\end{equation*}
Note that since $k\in\{1,2,\dots,N\}$ we have
\begin{equation*}
0<\frac{Y}{2} \le \frac{\pi}{2}
\end{equation*}
and so 
\begin{equation*}
\left|\sin\left(\frac{Y}{2}\right)\right| \le \frac{Y}{2}
\le C\frac{k}{N}.
\end{equation*}
It follows that,
\begin{equation}\label{abs-alp} 
\begin{aligned}
|\alpha_k|\le &C(2N+1)^{(d-2)/2} \left|\frac{\sin(NY)}{\sin^2\left(\frac{Y}{2}\right)}\right|\\
\le & CN^{(d-2)/2} \left|\frac{1}{\sin^2\left(\frac{Y}{2}\right)}\right| \\
\le & CN^{(d+2)/2} \left(\frac{1}{k}\right)^2.
\end{aligned} 
\end{equation}  
 
Recall that $\lambda_k$ is eigenvalue of $\phi_k$. Since $\lam_0=0$, from \eqref{eigen-rel} and \eqref{phi-sin} we have 
\begin{equation} \label{lam-0}   
\lambda_{j\e_1}  = \lambda_{j} = \left(\frac{(2k-1)\pi}{2N+1} \right)^2.
\end{equation} 

Using \eqref{0-coef}, \eqref{abs-alp} and \eqref{lam-0} we get  
\begin{align*} 
\sum_{j \in S_N^1\setminus\{0\}}\alpha_j^2   \lambda_{j\e_1}  & \leq C \sum_{j =1}^N  N^{d+2} \left(\frac{1}{j}\right)^4      \left(\frac{(2j-1)\pi}{2N+1} \right)^2 \\
& \leq CN^d \sum_{j =1}^N  \frac{1}{j^2}  \\
&  \leq C N^d. 
 \end{align*} 
 \end{proof}

 \section{Large distance tail estimates} \label{sec-large-t}  

We define 
\begin{equation*}
\bar u^{(i)} = \max_{\w,\z\in S_N}| u^{(i)}(\z) - u^{(i)}(\w) |.
\end{equation*}

Assume first that $d=2$.  Let $\alpha>0$, then from \eqref{eq:def-Q} we have 
\begin{equation}  \label{q-n-k} 
\log Q_N (R_N>  \alpha N( \log N)^{3/2} ) \leq \log P(R_N>\alpha N (\log N)^{3/2} ) -\log Z_{N}. 
\end{equation} 
Note that 
\begin{equation}  \label{r-bnd-t}
\begin{aligned} 
P(R_N>\alpha N (\log N)^{3/2})  
&\leq 2P\big( \cup_{i=1}^2 \{ \bar u^{(i)}>\alpha N (\log N)^{3/2}/2 \}  \big)  \\
&\leq 4 P\big( \bar u^{(1)}>\alpha N(\log N)^{3/2}/2\big) , 
\end{aligned} 
\end{equation} 
where we used the fact that $\bar u^{(i)}$ are i.i.d..

We will use the following standard bound on the tail distribution of the 
Gaussian random variable $W \sim N(0,\sigma^2)$, 
\begin{equation*}
P(W>a) \leq \frac{\sigma}{\sigma+a} \exp\left(\frac{-a^2}{2\sigma^2} \right). 
\end{equation*}

Using this bound and Proposition \ref{lem-cov}(i) we have for any $K \geq 1$, 
\begin{align*} 
P\big(  | u^{(i)}(\z) & - u^{(i)}(\w) | >\beta^{-1/2} KN(\log N)^{3/2} \big) \\
& \leq  C\exp\left\{-\frac{K^2N^2 (\log N)^3}{2c_1(\log N)^2} \right\}\\
 & \leq  C  \exp\left\{- \tilde c_1 K^2N^2 \log N  \right\}.
\end{align*} 
Note that this bound is uniform in $\z,\w \in S^2_N$. 

It follows that 
\begin{align*} 
&P\big(  \bar u^{(i)} >\beta^{-1/2} KN(\log N)^{3/2}/2 \big) \\
&\leq \sum_{\w, \z \in S^2_N} P\big(  |u^{(1)}(\z)-u^{(1)}(\w)|> \beta^{-1/2}  KN(\log N)^{3/2}/2\big)   \\ 
&\leq  C(2N+1)^4\exp\left\{- \frac{\tilde c_1}{4} K^2N^2 \log N \right\}\\
&\leq  C \exp\left\{- \frac{ \tilde c_1}{8} K^2 N^2 \log N   \right\}. 
\end{align*} 
Together with \eqref{r-bnd-t} we have 
\begin{equation*}
P(R_N>   \alpha  \beta^{-1/2}  N (\log N)^{3/2} )  \leq  C \exp\left\{-\frac{\tilde c_1\alpha^2N^2 \log N}{8} \right\}. 
\end{equation*}
 Using this bound together with Proposition \ref{prop-z}(i) and \eqref{q-n-k} we get for $d=2$, and $\alpha \geq 1$, 
\begin{align*} 
&\log Q_N (R_N>  \alpha \beta^{-1/2}  (\beta+\gamma)^{1/2} N(\log N)^{3/2}) \\
 & \leq \log P(R_N> \alpha\beta^{-1/2} (\beta+\gamma)^{1/2} N \log N) -\log Z_{N}  \\
 &\leq- (\beta+\gamma)  N^2 \log N \big( c_3    \alpha^2  -c_4  \big). 
 \end{align*} 
We then can choose $\alpha$ to be large enough to get the large distance tail estimate in Theorem \ref{th:main}.  
The proof for $d\geq 3$ follows similar lines, only now we use Proposition \ref{prop-z}(ii) and Proposition \ref{lem-cov}(ii).
\section{Small distance tail estimates} \label{sec-small} 
Let $\eps>0$, then from \eqref{eq:def-Q} we have the following: 
\begin{equation}  \label{d1} 
\begin{aligned} 
 \log Q_N (R_N <  \eps N) & \leq \log E\left[\exp\left\{-\gamma \int_{ \mathbb{R}^d}\ell_N(\y,\bu)^2d\y\right\}\mathbf{1}_{\{R <\eps N \}} \right] \\
 & \quad -\log Z_{N}.  
\end{aligned}  
\end{equation} 
Let 
\begin{equation} \label{j-def} 
\tilde J:= E\left[\exp\left\{-\gamma\int_{\mathbb{R}^d}\ell_N(\y,\bu)^2d\y\mathbf{1}_{\{R <\eps N \}}\right\} \right]. 
\end{equation} 
Note that on $\{R <\eps N \}$ we have 
\begin{equation} \label{gh1} 
\begin{aligned} 
 \int_{\mathbb{R}^d}\ell_N(\y,\bu)^2d\y &= 2^dN^{d} \eps^d \int_{-\eps N}^{\eps N}\dots\int_{-\eps N}^{\eps N}\ell_N(\y,\bu)^2 \frac{1}{2^d N^d\eps^d}d\y\\
&\geq 2^dN^d \eps^d \left(\int_{-\eps N}^{\eps N}\dots\int_{-\eps N}^{\eps N}\ell_N(\y,\bu) \frac{1}{2^d N^d\eps^d}d\y\right)^2 \\
&= \frac{1}{2^d N^d \eps^d} \left(\int_{-\eps N}^{\eps N}\dots\int_{-\eps N}^{\eps N}\ell_N(\y,\bu) d\y\right)^2, 
\end{aligned} 
\end{equation} 
where we used Jensen's inequality. 
Since on $\{R_N <\eps N \}$ we have 
\begin{equation*}
\int_{-\eps N}^{\eps N}\dots\int_{-\eps N}^{\eps N}\ell_N(\y,\bu) d\y= |S_N|= (2N+1)^d, 
\end{equation*}
and together with \eqref{gh1} we get that  
\begin{equation} \label{l-lower} 
\int_{\mathbb{R}^d}\ell_N(\y,\bu)^2d\y  \geq  \frac{2^dN^d}{ \eps^d}.
\end{equation} 
From \eqref{j-def} and \eqref{l-lower} we have

\begin{equation}  \label{small}
\tilde J \leq e^{ -\gamma \frac{ 2^dN^d}{ \eps^d}}. 
\end{equation} 
Together with \eqref{d1}, \eqref{j-def}, \eqref{small} with $\eps (\log N)^{-1/2}$ instead of $\eps$, and Proposition \ref{prop-z}(i) we get for $d=2$, 
\begin{align*}
\log &  Q_N (R_N < \eps \gamma^{1/2}  (\beta +\gamma)^{-1/2} N( \log N)^{-1/2} ) \\
 & \leq -  (\beta +\gamma) \left(   \frac{4N^2 \log N}{ \eps^2} -  C N^{2}\log N \right). 
\end{align*} 
By choosing $\eps>0$ small enough it follows that 
\begin{equation*}
\lim_{N\rightarrow \infty}  \log  Q_N (R_N < \eps \gamma^{1/2} (\beta +\gamma)^{-1/2} N(\log N)^{-1/2}) =0. 
\end{equation*}
Repeating the same steps as in the case where $d=2$ gives the following result for $d\geq 3$, 
\begin{equation*}
 \log  Q_N \left(R_N < \eps \gamma^{1/d} (\beta +\gamma)^{-1/d} N^{}  \right)   \leq -(\beta +\gamma) \left(   \frac{N^d }{ \eps^2} -  C N^{d}  \right).
\end{equation*}
Then choosing $\eps>0$ sufficiently small and taking the limit where $N\rightarrow \infty$ completes the proof of Theorem \ref{th:main}. 

\section{Proof of Proposition \ref{lem-cov}} \label{sec-pf-prop-var}
\begin{proof}[Proof of Proposition \ref{lem-cov}] We prove the result for the case where $\beta=1$. The extension for any $\beta>0$ follows from \eqref{u-eigen} by scaling.
We start with the proof of the lower bound. 
We want to show that for $d\ge2$, there exists a constant $C=C(d)$ such that 
for all $\w,\z\in S_N^d$, $\w\not=\z$, $i=1,\dots,d$ we have
\begin{equation*}
C(d) \leq \text{Var}(u^{(i)}(\z)-u^{(i)}(\w)).
\end{equation*}
Since $(u^{(i)})_{i=1,\dots,d}$ are i.i.d., we will omit the superscript $i$ for ease of notation.  
Let 
$\mathcal{F}_\z$ be the $\sigma$-field generated by $\{u(\bv): \, \bv \in S^d_N \setminus \{\z\}\}$ and define 
\[
\hat  u(\z) = E[u(\z)| \mathcal{F}_\z ]. 
\]
By using the conditional expectation projection theorem and then conditioning on $\mathcal{F}_\z$ we get 
\begin{align*}
\text{Var}(u^{}(\z)-u^{}(\w)) &= E\left[(u^{}(\z)-u^{}(\w))^2     \right] \\ 
&\geq E\left[(u^{}(\z)-\hat u(\z))^2   \right] \\ 
&= E\left[ E\left[(u^{}(\z)-\hat u(\z) )^2  \big|  \mathcal{F}_\z \right] \right]. 
\end{align*} 
Hence it is enough to show that there exists a constant $C(d)>0$ not 
depending on $\z$ such that  
\begin{equation} \label{eq:to-show}
\text{Var}[u(\z)| \mathcal{F}_\z] =  E\left[(u^{}(\z)-\hat u(\z) )^2  \big|  \mathcal{F}_\z \right]   \ge C(d).
\end{equation}

We consider the nearest neighbor values of $u(\z)$, which we denote by $\{u(\y)$:  $\y\sim \z \}$. These values are fixed once we condition on 
 $ \mathcal{F}_\z$. We further denote by $\mathcal N(\z)$ the number neighboring sites of $\z$. Note that for any $d \geq 2$ we have $ d \leq \mathcal N(\z) \leq 2d $. Then the part of the exponent of \eqref{eq:prob-density} which is
relevant to \eqref{eq:to-show} is
\begin{equation*}
\begin{split}
\sum_{\y\sim \z}&(u(\z)-u(\y))^2  = \mathcal{N} (\z) \cdot u(\z)^2 + 2u(\z)\sum_{\y\sim \z}u(\y) +C,
\end{split}
\end{equation*}
where $C$ is $\mathcal{F}_\z$-measurable.  By completing the squares we get 
\begin{equation} \label{cond-var} 
\begin{split}
\sum_{\y\sim \z}(u(\z)-u(\y))^2   &=  \mathcal{N} (\z) \left(u(\z)-\frac{1}{\mathcal N(\z)}\sum_{\y\sim \z}
    u(y)\right)^2 + C' \\
&= \mathcal{N} (\z) \left(u(\z)-C''\right)^2 + C',
\end{split}
\end{equation}
where again $C'$ and $C''$ are $\mathcal{F}_\z$-measurable.  It follows from the equality in \eqref{eq:to-show} and from \eqref{cond-var} that conditioned on $\mathcal{F}_\z$, $u(\z)$ is a Gaussian random variable with variance $(2\mathcal N(\z))^{-1} > (4d)^{-1}$ and therefore \eqref{eq:to-show} follows.

Next we prove the upper bound. Recall that $\mathbf{V}$ was defined in 
\eqref{v-def}. Let $\mathbf{e}_\x - \mathbf{e}_\y\in\mathbf{V}$. Recall that 
$u$ has density function given by \eqref{eq:prob-density}. Let 
$\X=\{\X_t\}_{t\geq 0}$ be the continuous time Markov chain associated with 
$\Delta$ and $\{ \mathcal{P}_t \}_{t\geq 0}$ the corresponding probability 
transition function.

It follows that the variance of $u(\x)-u(\y)$ is given in 
terms of $\Delta^{-1}$ as follows:
\begin{equation}  \label{eq:variance-integral}
\begin{split}
\text{Var}(u(\x)-u(\y)) 
&= \left\langle(\mathbf{e}_\x-\mathbf{e}_\y),
         \Delta^{-1}(\mathbf{e}_\x-\mathbf{e}_\y)\right\rangle \\
&= \int_{0}^{\infty}
      \left\langle(\mathbf{e}_\x-\mathbf{e}_\y),
         \mathcal{P}_t(\mathbf{e}_\x-\mathbf{e}_\y)\right\rangle dt. 
\end{split}
\end{equation}

\subsection{Estimation of the integrand in \eqref{eq:variance-integral} for small $t$}

We consider the case where $t\leq K_0N^2\log N$ for some constant 
$K_0>0$ to be determined. 
Note that we can extend $\mathcal{P}_t$ to a semigroup on all of 
$\mathbf{V}^+$, corresponding to the same random process $\X_t$.  We write 
$P_\x$ to indicate the starting point $\X_0=\x$.  
 
Next, we note that the components $X^{(k)}$, $k=1,\dots,d$ of $\X$ are 
independent, since their jump times are independent Poisson processes which determine the jumps of $\X$.  
Let $\x=(x_1,\dots,x_d) \in S_N^d$, then
\begin{equation} \label{proj-trans} 
\left\langle\mathbf{e}_x,\mathcal{P}_t\mathbf{e}_x\right\rangle 
= P_\x[\X_t=\x]
= \prod_{k=1}^d P_{x_k}(X^{(k)}_t=x_k).
\end{equation}
In what follows we focus on the marginal distribution of $X^{(k)}$. We write $z=x_k$, and $Z_t=X^{(k)}_t$ and get
\begin{equation} \label{eq:prob-return}
P_z(Z_t=z) = \sum_{n=0}^{\infty}P(T_t=n)P_z(S_n=z),
\end{equation}
where $T_t$ Poisson process with intensity $1/d$ and $S_n$ is a discrete-time 
nearest-neighbor one-dimensional simple random walk with reflection at the 
boundary $\pm N$.  

Let $Y$ be a Poisson random variable with mean $\lambda$ then using Markov inequality we get for all $\theta>0$ and $y> \lambda$,  
\begin{equation} \label{pois-ineq} 
\begin{aligned} 
P(Y> y) &\leq \frac{E[e^{\theta Y}]}{e^{\theta y}} \\
&\leq \frac{(e\lambda)^ye^{-\lambda}}{y^y},
\end{aligned} 
\end{equation} 
where we choose $\theta=\log(y/\lambda)>0$ in the second inequality.  

Let $K_1 > eK_0$. Recall that $t\leq K_0N^2\log N$ and that $T_t$ is a Poisson random variable with mean $(2d)^{-1}t$. From \eqref{pois-ineq} we get the following bound on the tail 
distribution of $T_t$, 
\begin{equation} \label{pois-bnd}
\begin{aligned} 
P(T_t > K_1N^2\log N )& \leq  \frac{e^{- (2d)^{-1}t}((2d)^{-1} te)^{K_1N^2\log N}}{(K_1N^2\log N)^{K_1N^2\log N}} \\ 
&\leq e^{-c N^2\log N},  
\end{aligned} 
\end{equation} 
for some constant $c>0$.

Using the reflection principle we note that  
\begin{align} \label{refl} 
P_z(S_n=z) = 
  \sum_{k\in \mathbb {Z} } P_z\left(  W_n = 2Nk+(-1)^k z  \right),
\end{align} 
where $\{W_n\}_{n\geq 1}$ is a discrete-time simple random walk on $\mathbb{Z}$. 

Since $P_z(W_n=z)=0$ if $n$ is odd, we need to take into account only even number of steps, so we have 
$$
P_z(W_{2m}=z)   = \binom{2m}{m}2^{-2m}. 
$$
Note that the transition probability from $z$ to all other points~$2Nk+(-1)^k z$ in the right-hand side of \eqref{refl} can be computed according to the same binomial distribution, since these transitions also require an even number of steps. Moreover the maximum of the above binomial distribution (i.e. Bin$(2m, 1/2)$) is attained at $m$ so we must have   
\begin{equation} \label{up-ref-p} 
P_z\left(  W_{2m} = 2Nk+(-1)^k z  \right) \leq P_z(W_{2m}=z), \quad \textrm{for all } k \in \mathbb Z. 
\end{equation} 
Let $\ell(N)$ be the number of points from the set $\{ 2Nk+(-1)^k z: z\in \mathbb{Z} \}$ visited by $W_n$ up to $n = [K_1N^2\log N]$. We will use a special case of Corollary A.2.7 in \cite{Lawler2010} which states that there exist constants $C_1,c_2>0$ such that 
$$
P\big(\max_{i=0,\dots,n} |W_n| >s\sqrt{n}\big) \leq C_1e^{-c_2s^2},  \quad \textrm{for all }  n\geq0, \ s>0,
$$
in order to bound the tail probability of $\ell(N)$ as follows. Let $K_3>0$ then we have 
\begin{equation} \label{l-n-bnd} 
\begin{aligned} 
P(\ell(N) > K_3 \log N) &\leq P\left(\sup_{0\leq i \leq K_1N^2\log N}  |W_i| >   \frac{K_3}{4} N \log N  \right) \\ 
&\leq C_1e^{-c_2 (K_3^2/(16K_2)) \log N} \\
& \leq C_1N^{-c_2K_3^2/(16K_2)}.
\end{aligned} 
\end{equation} 

We would like to bound \eqref{eq:prob-return} on the event $T_t \leq K_1N^2\log N$. From \eqref{refl} we have 
\begin{equation}   \label{split-j} 
\begin{aligned}
&\sum_{m=0}^{K_1N^2\log N}P(T_t=2m)P_z(S_{2m}=z)  \\
 & =\sum_{m=0}^{K_1N^2\log N}P(T_t=2m)  \sum_{k\in \mathbb {Z} } P_z\left(  W_{2m} = 2Nk+(-1)^k z  \right) \\ 
&=\sum_{m=0}^{K_1N^2\log N}P(T_t=2m) \sum_{|k| \leq 2^{-1}K_3 \log N  } P_z\left(  W_{2m} = 2Nk+(-1)^k z  \right)  \\
&\qquad +\sum_{m=0}^{K_1N^2\log N}P(T_t=2m) \sum_{|k| > 2^{-1}K_3 \log N} P_z\left(  W_{2m} = 2Nk+(-1)^k z  \right)  \\
&:=J_1 +J_2. 
\end{aligned}
\end{equation}

From \eqref{l-n-bnd} it follows that $J_2$ is bounded by 
 \begin{equation}   \label{old} 
\begin{split}
J_2&\leq \sum_{m=0}^{K_1N^2\log N} P(T_t=2m) P(\ell(N) > K_3 \log N)  \\
&\leq C_1N^{-c_2K_3^2/(16K_2)}\sum_{m=0}^{\infty}P(T_t=2m)    \\
&\leq   C N^{-r}. 
\end{split}
\end{equation}
for $r>0$ to be determined. Note that the last inequality follows by choosing $K_3$ large enough.

Using \eqref{up-ref-p} and we get for $J_1$ that
 \begin{equation} \label{j-1-1}   
 J_1\leq K_3 \log N  \sum_{m=0}^{K_1N^2\log N} P(T_t=2m)P_z\left(  W_{2m} =  z  \right).
 \end{equation}

We estimate the above sum as follows,  
\begin{equation} \label{prob-calc} 
\begin{split}
 \sum_{m=0}^{\infty}P(T_t=2m)P_z(W_{2m}=z) 
 &= \sum_{m=0}^{\infty}\frac{(t/d)^{2m}}{(2m)!}e^{-t/d}
           \cdot\binom{2m}{m}2^{-2m}  \\
&= e^{-t/d}\sum_{m=0}^{\infty}\frac{(t/(2d))^{2m}}{(m!)^2}. 
\end{split}
\end{equation}

Recall that
\begin{equation*}
I_0(2z) = \sum_{m=0}^{\infty}\frac{z^{2m}}{(m!)^2},
\end{equation*}
where $I_0$ is a modified Bessel function of the first kind,
so for large values of $y$, $I_0$ has the following asymptotics, 
\begin{equation*}
I_0(y) \sim \frac{e^y}{\sqrt{2\pi y}}
\end{equation*}
and so by taking $y=2z=2t/(2d)=t/d$, it follows that there exists a constant $C>0$ not depending on $z$ such that  
\begin{equation} \label{new} 
\begin{aligned} 
\sum_{m=0}^{\infty}P(T_t=2m)P_z(W_{2m}=z)  
&\leq C e^{-t/d}\frac{e^{t/d}}{\sqrt{2\pi t/d}} \\
& \leq  Ct^{-1/2}, \quad \textrm{  for all } t\geq 1. 
\end{aligned} 
\end{equation}
From \eqref{split-j}--\eqref{j-1-1} and \eqref{new} it follows that 
\begin{equation*}   
\begin{aligned}
&\sum_{m=0}^{K_1N^2\log N}P(T_t=2m)P_z(S_{2m}=z) \leq C\log N t^{-1/2} + N^{-r}.
\end{aligned}
\end{equation*}
Using this bound together with \eqref{eq:prob-return} and \eqref{pois-bnd} we have  
\begin{equation}   \label{p-z-bnd} 
\begin{split}
P_z(Z_t=z) &\leq  C_1\log N t^{-1/2} +C_2 N^{-r} +  e^{-c N^2\log N} \\  
&\leq C(\log N t^{-1/2} +N^{-r}). 
 \end{split}
\end{equation}

From \eqref{proj-trans} and \eqref{p-z-bnd} we get  
\begin{equation}\label{small-t-b} 
\begin{aligned}
&|\left\langle(\mathbf{e}_x-\mathbf{e}_y),
         \mathcal{P}_t(\mathbf{e}_x-\mathbf{e}_y)\right\rangle |
\\
&\le 2 \left\langle\mathbf{e}_x,\mathcal{P}_t\mathbf{e}_x\right\rangle 
     + 2 \left\langle\mathbf{e}_y,\mathcal{P}_t\mathbf{e}_y\right\rangle \\
     &\leq C(d) (\log N t^{-d/2} +N^{-rd}) , \quad \textrm{  for all } 0\leq t\leq K_0N^2\log N. 
\end{aligned}
\end{equation}  
From \eqref{small-t-b} and by choosing $r$ sufficiently large we finally get, 
 \begin{equation}\label{int-small} 
\begin{aligned} 
\int_{1}^{K_0N^2\log N}  |\left\langle(\mathbf{e}_x-\mathbf{e}_y),
         \mathcal{P}_t(\mathbf{e}_x-\mathbf{e}_y)\right\rangle|dt & \leq   
            \begin{cases} C(\log N)^2 & \textrm{ if } d=2, \\
            C  &\textrm{ if } d\geq 3.  
         \end{cases} 
       \end{aligned} 
\end{equation}

\subsection{Estimation of the integrand in \eqref{eq:variance-integral} for 
large $t$}

Here we consider the case where $t\ge KN^2\log N$ with $K$ 
sufficiently  large. We again consider the one-dimensional projection of the 
continuous time random walk $Z$ reflected at $\pm N$, as in 
\eqref{eq:prob-return}. We expect that in this case there is a large 
probability of $Z$ reaching the boundary on or before time $t$, hence we 
expect the semigroup $\mathcal{P}_t$ to even out any given function. Thus we 
keep both $\mathbf{e}_x,\mathbf{e}_y$ in 
\eqref{eq:variance-integral}, use the fact that  
$\langle\mathbf{e}_x-\mathbf{e}_y,\mathcal{P}_t(\mathbf{e}_x-\mathbf{e}_y)\rangle$ 
mostly cancels out.  
We therefore show that $
P_x\left(Z_t=x\right)-P_y\left(Z_t=x\right)$ is small for large $t$ uniformly in $x,y$.  

We will use the coupling method in order to bound the difference in the probabilities above. To this end, we construct two i.i.d copies 
$Z^{(1)},Z^{(2)}$ of the process $Z$ with $Z^{(1)}_0=x$ and $Z^{(2)}_0=y$ on the same probability space. We seek a (random) coupling time 
$\tau$ such that if $t>\tau$ then $Z^{(1)}_t=Z^{(2)}_t$. In that case we 
would have
\begin{equation} \label{coupling-eq} 
\left|P_x\left(Z_t=x\right)-P_y\left(Z_t=x\right)\right| \le P(\tau> t).
\end{equation}

Next we derive an upper on $P(\tau>t)$. For any process $Y$, and for $z \in S_N^1$ define   
$$
 \tau^{Y}_{z} = \inf\{t \geq 0: Y_t =z\}.
$$
Without loss of generality assume that $x>y$. We observe that in this case we have for $Z^{(1)}$ and $Z^{(2)}$ as above that $\tau \leq \tau^{Z^{(1)}}_{-N}$.  
Using reflection and translation invariance we get
\begin{equation}\label{tau-bnd} 
\begin{aligned} 
P(\tau>t) &\leq P_x(\tau^{Z^{(1)}}_{-N} >t)  \\
&\leq P_{N-x} ( \tau^{Y}_{2N} >t) \\
&\leq P_{0} ( \tau^{Y}_{2N} >t), 
\end{aligned} 
\end{equation} 
where $Y=\{Y_t\}_{t\geq 0}$ is a continuous time simple random walk on $\mathbb{Z}$, reflected at $0$ with jumps rate similar to $Z^{(1)}$. 

Recall that $W$ as simple random walk on $\mathbb{Z}$ and $\{T_t\}_{t\geq0}$ is a Poisson process with intensity $1/d$. 
By Proposition 2.4.5 in \cite{Lawler2010} we get that there exists constants $C_1, C_2>0$ such that for all integer $n>0$ and any $r>0$ we have  
$$
P_0( \tau^{|W|}_{n}> r n^2 )  \leq C_1e^{-C_2r}.
$$
We therefore get  
\begin{equation}\label{tau-y-bnd} 
\begin{aligned} 
P_{0} ( \tau^{Y}_{2N} >t)& = \sum_{m=0}^{\infty}P(T_t=m)P_0( \tau^{|W|}_{2N}> m  )  \\
&\leq  C_1 \sum_{m=0}^{\infty}e^{-t/d}\frac{(t/d)^{m}}{(m)!}e^{-C_2mN^{-2}} \\ 
&=C_1 E[e^{-C_2N^{-2} T_t}] \\
&= C_1 \exp\left \{\frac{t}{d}(e^{-C_2N^{-2}  }-1) \right\} \\
&\leq \tilde C_1 e^{-\tilde C_2 tN^{-2}},   
\end{aligned}  
\end{equation} 
where we have used the expression for the characteristic function of $T_t$. 
Combining \eqref{coupling-eq}, \eqref{tau-bnd} and \eqref{tau-y-bnd} gives 
\begin{equation}  \label{prob-dif-b} 
\begin{split}
\left|P_x\left(Z_t=x\right)-P_y\left(Z_t=x\right)\right| 
     & \leq \tilde C_1 e^{-\tilde C_2 tN^{-2}},   \\
& \textrm{for all } x,y \in [-N,N], \ t>0. 
\end{split}
\end{equation} 
From \eqref{proj-trans} we get, 
\begin{equation}\label{large-t} 
\begin{aligned} 
\left|\left\langle\mathbf{e}_\x,\mathcal{P}_t\mathbf{e}_\x\right\rangle -\left\langle\mathbf{e}_\y,\mathcal{P}_t\mathbf{e}_x\right\rangle \right| 
&= \left| P_\x[\X_t=\x] - P_\y[\X_t=\x] \right|\\ 
&=\left|  \prod_{k=1}^d P_{x_k}(X^{(k)}_t=x_k)- \prod_{k=1}^d P_{y_k}(X^{(k)}_t=x_k)  \right| \\ 
&\leq   C(d) e^{-\tilde C_2 tN^{-2}},
\end{aligned} 
\end{equation}
where we used \eqref{prob-dif-b} and the triangle inequality in the last inequality.  

 Then by choosing $K_0$ large enough, we get from \eqref{large-t} that 
 \begin{equation}\label{int-large} 
\begin{aligned} 
\int_{K_0N^2\log N}^{\infty} |\left\langle(\mathbf{e}_\x-\mathbf{e}_\y),
         \mathcal{P}_t(\mathbf{e}_\x-\mathbf{e}_\y)\right\rangle|dt & \leq    C(d)\int_{K_0N^2\log N}^{\infty}  e^{-\tilde C_2 tN^{-2}}dt \\ 
&\leq C N^2 N^{-\tilde C_2 K_0} \\ 
& \leq \tilde C. 
       \end{aligned} 
\end{equation}
Finally from \eqref{int-small}, \eqref{int-large} and by noting that the integral on the right-hand side of \eqref{eq:variance-integral} is bounded trivially by a constant on the integration region $[0,1]$ we get the upper bound in Proposition \ref{lem-cov}. 
\end{proof}

 \appendix
\label{appendix}
\section{Some heuristic ideas} 
We introduce some heuristic ideas to shed light on our main results. 
We consider parameters $i,j\in\{-N,\ldots,N\}$, and let 
$\X_{i,j}\in\mathbb{R}^2$. We define the energy of for the field $\bold X$ as
\begin{equation*}
E_N(\X)=\frac{1}{2}\sum_{|(i,j)-(i',j')|=1}|\X_{i,j}-\X_{i',j'}|^2
\end{equation*}
Then we define $P_N$ to be a constant times $\exp(-E_N(\X))$.

For $\x\in\mathbb{R}^2$, let $\ell_N(\x)$ be the number of points $\X_{i,j}$
within distance $1/2$ of $\x$. We weight the probability $P_N$ by 
$\exp(-\gamma\mathcal{E}_N(\X))$ where
\begin{equation*}
\mathcal{E}_N(\X)=\int_{\mathbb{R}^2}\ell_N(\x)^2d\x.
\end{equation*}

Now we argue heuristically, and suppose that the $\X_{i,j}$ are evenly spread
in a ball of radius $R$. One way to spread them evenly is to let
\begin{equation*}
\X_{i,j}\approx\frac{R}{N}\cdot(i,j)
\end{equation*}
In that case,
\begin{equation*}
E_N(\X)\approx C\left(\frac{R}{N}\right)^2N^2=CR^2.
\end{equation*}
Also, with the hypothesis of even spreading, we have that either $\ell(\x)=0$
or
\begin{equation*}
\ell_N(\x)\approx C\frac{N^2}{R^2}
\end{equation*}
and then
\begin{equation*}
\mathcal{E}_N(\X)\approx C\left(\frac{N^2}{R^2}\right)^2R^2 
=C\frac{N^4}{R^2}
\end{equation*}

Equating $E_N(\X)$ and $\mathcal{E}_N(\X)$, we get
\begin{equation*}
R^2=\frac{N^4}{R^2}
\end{equation*}
and so 
\begin{equation*}
R=N.
\end{equation*}
\\
\section*{Acknowledgments}
We are very grateful to the Associate Editor and to the anonymous referee for careful reading of the manuscript and for a number of useful comments and suggestions that significantly improved this paper.
\\

\textbf{Funding.} The work of Carl Mueller is partially supported by the Simons grant 513424.
\medskip  \\
\textbf{Availability of data and material.} Not applicable.
\medskip \\
\textbf{ Compliance with ethical standards.} The authors have no conflicts of interest to declare that are relevant to the content of this article.
\medskip \\
\textbf{Code availability.} Not applicable.

%\bibliographystyle{plain}
%\printindex
%\bibliography{GFF}

\end{document}